\documentclass[final,hidelinks]{siamart251216}

\usepackage[T1]{fontenc}
\usepackage{lmodern,amssymb,mathtools,bm,booktabs,microtype,placeins}

\allowdisplaybreaks[1]
\graphicspath{{figures/}}
\headers{Boundary Collisions and Weak Excitation in LQR}{Kang Liu}
\title{Boundary Collisions and Weak Excitation in LQR Policy-Gradient Flows}
\author{Kang Liu\thanks{School of Future Technology, Xi'an Jiaotong University, Xi'an, China (\email{kanyo@foxmail.com}).}}
\hypersetup{pdftitle={Boundary Collisions and Weak Excitation in LQR Policy-Gradient Flows},pdfauthor={Kang Liu}}

\begin{document}
\maketitle
\begin{abstract}
Closed-loop stability concerns all initial states, whereas an LQR cost evaluates only the directions excited by a prescribed initial-state distribution. We study how this distinction affects Euclidean policy-gradient flows. For every nonzero deficient covariance rank, we construct relatively open families of controllable data and open sets of stabilizing initial gains whose exact gradient trajectories reach the stability boundary in finite optimization time. The state Gramian can remain positive definite at every time before collision. By separating the cost contribution of a critical mode, we obtain a local collision criterion based on direction attraction and motion toward the stability boundary. Adding excitation of strength $\varepsilon$ yields matching $\Theta(\sqrt\varepsilon)$ bounds on the local minimum stability margin for trajectories that would otherwise collide. For a class of single-input, corank-one systems, these bounds hold over the entire trajectory without a common compact-set assumption. Numerical examples further illustrate the dynamics and scaling laws.

\end{abstract}

\section{Introduction}\label{sec:introduction}

Policy optimization for the linear quadratic regulator (LQR) reduces control cost by updating a stabilizing feedback gain. Closed-loop stability requires the system to return to equilibrium from every initial state. The control cost, however, is evaluated under a prescribed initial-state distribution. If that distribution is supported on a proper subspace, responses from other initial states are not directly included in the objective. We study how this distinction affects policy optimization: does updating the gain along the negative cost gradient preserve closed-loop stability?

Consider the continuous-time system $\dot x(t)=Ax(t)+Bu(t)$ with state feedback $u(t)=-Kx(t)$, where $x(t)\in\mathbb R^n$ is the state, $u(t)\in\mathbb R^m$ is the input, $A\in\mathbb R^{n\times n}$, $B\in\mathbb R^{n\times m}$, and $K\in\mathbb R^{m\times n}$. Assume that $(A,B)$ is stabilizable. Write $F_K:=A-BK$ and let $\mathcal S:=\{K:F_K\text{ is Hurwitz}\}$ be the set of stabilizing gains; a matrix is Hurwitz if all its eigenvalues have negative real parts. For symmetric positive definite weights $Q\in\mathbb R^{n\times n}$ and $R\in\mathbb R^{m\times m}$, and a zero-mean initial state with covariance $\Sigma\succeq0$, define
\[
J_\Sigma(K):=\mathbb E\int_0^\infty
x(t)^\top(Q+K^\top RK)x(t)\,\mathrm dt,
~~ K\in\mathcal S.
\]
The expectation is over the initial state, and $\Sigma$ specifies the directions and strengths of the initial excitation. We consider the Euclidean policy-gradient flow $K'(\tau)=-\nabla_KJ_\Sigma(K(\tau))$ from $K(0)=K_0\in\mathcal S$. Here $\tau$ is optimization time, distinct from physical time $t$, and the gradient is defined by the Frobenius inner product. As long as the trajectory remains in $\mathcal S$, its cost satisfies $\frac{\mathrm d}{\mathrm d\tau}J_\Sigma(K(\tau))=-\|\nabla_KJ_\Sigma(K(\tau))\|_F^2\le0$, where $\|\cdot\|_F$ is the Frobenius norm.

When $\Sigma\succ0$, cost sublevels are compact subsets of $\mathcal S$. Descent therefore keeps the gradient trajectory away from the stability boundary and ensures convergence to the Riccati gain. With rank-deficient covariance, a closed-loop mode can approach instability while its excitation vanishes, allowing the cost to remain bounded. A bounded-cost sequence near the boundary does not, by itself, determine the behavior of a gradient trajectory. Our first question is whether the exact negative-gradient flow can reach the stability boundary in finite optimization time, and whether this behavior persists under small perturbations of the system data and initial gain.

To excite the missing directions, consider $\Sigma_\varepsilon:=\Sigma+\varepsilon H$, where $H\succeq0$ is fixed, $\Sigma+H\succ0$, and $\varepsilon>0$ is the added excitation strength. Each $\Sigma_\varepsilon$ is positive definite, so the corresponding gradient flow preserves closed-loop stability for all optimization times and converges. This fixed-$\varepsilon$ guarantee does not quantify how closely the trajectory approaches the stability boundary. We measure this proximity by the spectral stability margin
\[
\mathfrak m(K):=-\max_{\lambda\in\operatorname{spec}(F_K)}\operatorname{Re}\lambda,
\]
where $\operatorname{spec}(F_K)$ is the closed-loop spectrum and $\operatorname{Re}\lambda$ denotes the real part. For $K\in\mathcal S$, the positive quantity $\mathfrak m(K)$ is the distance from the spectrum to the imaginary axis. Let $K_\varepsilon$ denote the gradient trajectory of $J_{\Sigma_\varepsilon}$, with the same initial gain $K_\varepsilon(0)=K_0$ for all $\varepsilon$. Our second question concerns the minimum margin as $\varepsilon\downarrow0$.
Generally, this paper makes two contributions.

i)\emph{Robust finite-time boundary collisions at every nonzero deficient rank.}
For every $n\ge2$, $m\ge1$, and $1\le s<n$, we construct a nonempty relatively open set of data $(A,B,Q,R,\Sigma)$ with $(A,B)$ controllable, $Q\succ0$, $R\succ0$, and $\operatorname{rank}\Sigma=s$. Every member admits a nonempty open set of stabilizing initial gains for which
\[
K(\tau)\in\mathcal S~(0\le\tau<T),~~
K(\tau)\longrightarrow\bar K\in\partial\mathcal S
~\text{as }\tau\uparrow T<\infty.
\]
Here $\partial\mathcal S$ is the stability boundary and $\bar K$ is the collision gain. The matrix $F_{\bar K}$ has a simple zero eigenvalue, and all remaining eigenvalues have strictly negative real parts. Openness is relative to the data space with covariance rank fixed at $s$. The state Gramian, which accumulates the state second moment over physical time, can remain positive definite throughout the trajectory before collision, while its smallest eigenvalue vanishes linearly with the stability margin (\cref{thm:open-collision}). Thus neither cost descent nor pointwise positive definiteness of the state Gramian ensures stability throughout optimization.

2)\emph{Matching square-root bounds under weak excitation.}
For any fixed initial gain in a local collision basin with a strictly attracting direction, let $\mathfrak m_{\min}^{\mathrm{loc}}(\varepsilon)$ be the minimum margin over the first local interval after adding excitation. Then
\[
c\sqrt\varepsilon\le\mathfrak m_{\min}^{\mathrm{loc}}(\varepsilon)
\le C\sqrt\varepsilon
\]
for all sufficiently small $\varepsilon>0$, with positive constants $c,C$ independent of $\varepsilon$ (\cref{thm:local-sqrt}). For the single-input, corank-one members of the constructed collision families, namely $m=1$ and $s=n-1$, the same bounds hold over the entire trajectory. Rank-one excitation in the sole missing direction suffices, without a common compact-set assumption across $\varepsilon$ (\cref{thm:global-sqrt}). The largest control cost over unit-norm initial states during the first local interval is consequently $\Theta(\varepsilon^{-1/2})$, where $\Theta$ denotes matching bounds by positive constants independent of $\varepsilon$. Weak excitation therefore quantifies both the minimum stability margin and the worst initial-state cost encountered during optimization.

\Cref{sec:setting} gives the basic representations and the full-rank stability guarantee. \Cref{sec:collisions} develops the collision criterion and the open data families. \Cref{sec:excitation} establishes the local and whole-trajectory square-root bounds, and \cref{sec:numerics} presents numerical results. Related works are listed in \cref{app:relate} and supporting proofs are collected in \cref{app:proofs}.

\section{Problem setting and basic properties}\label{sec:setting}

We use the system and cost introduced above and assume that $(A,B)$ is stabilizable, $Q\succ0$, $R\succ0$, $\mathbb E[x_0]=0$, and $\mathbb E[x_0x_0^\top]=\Sigma\succeq0$. Throughout, $\Sigma\ne0$; if $\Sigma=0$, the cost and gradient vanish identically and the gain remains constant.

\subsection{Lyapunov representations of the cost and gradient}

For $K\in\mathcal S$, set $W_K:=Q+K^\top RK$ and define
\[
P_K:=\int_0^\infty e^{F_K^\top t}W_Ke^{F_Kt}\,\mathrm dt,
~~
X_{K,\Sigma}:=\int_0^\infty e^{F_Kt}\Sigma e^{F_K^\top t}\,\mathrm dt.
\]
These integrals converge because $F_K$ is Hurwitz, meaning that all its eigenvalues have negative real parts. The matrices $P_K\succ0$ and $X_{K,\Sigma}\succeq0$ are the unique symmetric solutions of
\begin{align}
F_K^\top P_K+P_KF_K+W_K&=0,\label{eq:setting-P}\\
F_KX_{K,\Sigma}+X_{K,\Sigma}F_K^\top+\Sigma&=0.\label{eq:setting-X}
\end{align}
For a deterministic initial state $x_0$, the accumulated control cost is $x_0^\top P_Kx_0$, so $P_K$ describes the cost in each initial-state direction. The state Gramian $X_{K,\Sigma}=\int_0^\infty\mathbb E[x(t)x(t)^\top]\,\mathrm dt$ is the accumulated second moment of the state after propagation of the initial excitation through the closed-loop dynamics. We abbreviate it as $X_K$ when the covariance is fixed.

The integral representations give
\(
J_\Sigma(K)=\operatorname{tr}(P_K\Sigma)
=\operatorname{tr}(W_KX_{K,\Sigma}),
\)
where $\operatorname{tr}$ denotes the trace. The cost is smooth on the open set $\mathcal S$, with Euclidean gradient
\begin{equation}
\nabla_KJ_\Sigma(K)
=2(RK-B^\top P_K)X_{K,\Sigma}.\label{eq:setting-grad}
\end{equation}
The derivation is given in Appendix~\ref{app:gradient}. Here the Frobenius inner product is $\langle M,N\rangle_F:=\operatorname{tr}(M^\top N)$. Equation~\eqref{eq:setting-grad} shows that initial excitation affects the gain update through the state Gramian. Even when $\Sigma$ is rank deficient, the closed-loop dynamics may propagate excitation to all state directions, giving $X_{K,\Sigma}\succ0$.

\subsection{Gradient trajectories and boundary collisions}

The gradient flow is
\begin{equation}
K'(\tau)=-2(RK(\tau)-B^\top P_{K(\tau)})X_{K(\tau),\Sigma},
~~ K(0)=K_0\in\mathcal S.\label{eq:gradient-flow}
\end{equation}
Primes on trajectory variables denote differentiation with respect to optimization time $\tau$. Smoothness of the vector field on $\mathcal S$ gives a unique maximal stabilizing trajectory $K:[0,T_{\max})\to\mathcal S$, where $T_{\max}\in(0,\infty]$ is the maximal time for continuation within $\mathcal S$. For $0\le\tau<T_{\max}$,
\begin{equation}
J_\Sigma(K(\tau))+
\int_0^\tau\|\nabla_KJ_\Sigma(K(s))\|_F^2\,\mathrm ds
=J_\Sigma(K_0).\label{eq:energy}
\end{equation}
Thus the trajectory remains in the sublevel set determined by its initial cost.

We say that a trajectory undergoes a \emph{finite-time boundary collision} if $T_{\max}=T<\infty$ and $K(\tau)\to\bar K\in\partial\mathcal S$ as $\tau\uparrow T$ for some finite gain $\bar K$. $T$ is the collision time.

\subsection{Stability under full-rank covariance}

Under the standing assumptions, the continuous-time algebraic Riccati equation
\[
A^\top P_\star+P_\star A-P_\star BR^{-1}B^\top P_\star+Q=0
\]
has a unique stabilizing solution $P_\star\succ0$~\cite{lancaster1995}. The corresponding Riccati optimal gain is $K_\star:=R^{-1}B^\top P_\star$, where ``stabilizing'' means that $A-BK_\star$ is Hurwitz. The following is a standard property of continuous-time LQR policy-gradient flows~\cite{bu2020}.

\begin{proposition}[Global convergence under positive definite covariance]\label{prop:full-rank}
If $\Sigma\succ0$, then for every $K_0\in\mathcal S$, the gradient flow~\eqref{eq:gradient-flow} exists for all $\tau\ge0$ and remains in a compact subset of $\mathcal S$. Moreover,
\[
\lim_{\tau\to\infty}K(\tau)=K_\star,
~~ \inf_{\tau\ge0}\mathfrak m(K(\tau))>0.
\]
\end{proposition}

\begin{proof}
When $\Sigma\succ0$, the cost tends to infinity as the gain becomes unbounded or as bounded gains approach the stability boundary. Hence every nonempty sublevel set $\mathcal L_c:=\{K\in\mathcal S:J_\Sigma(K)\le c\}$ is a compact subset of $\mathcal S$. By~\eqref{eq:energy}, the trajectory remains in $\mathcal L_{J_\Sigma(K_0)}$, which gives global continuation and a strictly positive lower bound on the stability margin.

Since $X_{K,\Sigma}\succ0$, every stationary gain satisfies $RK=B^\top P_K$. Substitution into~\eqref{eq:setting-P} shows that $P_K$ is the stabilizing Riccati solution, so $K=K_\star$. On the compact trajectory sublevel, the energy identity and uniform continuity of the gradient imply that the gradient norm tends to zero. Every accumulation point is therefore stationary, and uniqueness gives convergence. The coercivity estimates and the vanishing-gradient argument are detailed in Appendix~\ref{app:full-rank}.
\end{proof}

For rank-deficient covariance, sublevel sets may approach the stability boundary. We next introduce a local representation of the critical mode to analyze how such an approach can arise along a gradient trajectory.

\subsection{Local representation near a simple zero mode}\label{sec:local-spectrum}

Let $\bar K\in\partial\mathcal S$ be such that $F_{\bar K}$ has an algebraically simple zero eigenvalue, while all other eigenvalues lie strictly in the left half-plane. In a sufficiently small neighborhood $\mathcal U$ of $\bar K$, this simple eigenvalue extends smoothly as $-y(K)$, with $y(\bar K)=0$, and the other eigenvalues remain uniformly separated from the imaginary axis~\cite[Chapter II]{kato1995}. After shrinking $\mathcal U$,
\[
\mathcal U\cap\mathcal S=\{K\in\mathcal U:y(K)>0\},
~~ \mathfrak m(K)=y(K)~(K\in\mathcal U\cap\mathcal S).
\]
Thus $y$ is precisely the spectral stability margin on the local stable side.

Choose smooth real right and left eigenvectors $v(K)$ and $\ell(K)$ satisfying
\[
F_Kv(K)=-y(K)v(K),~~
\ell(K)^\top F_K=-y(K)\ell(K)^\top,
\]
normalized by $\|v(K)\|_2=1$ and $\ell(K)^\top v(K)=1$, where $\|\cdot\|_2$ denotes the Euclidean vector norm. The associated rank-one spectral projection is $\Pi(K):=v(K)\ell(K)^\top$. The initial coefficient of this mode is $\ell(K)^\top x_0$, with second moment $\mathbb E[(\ell(K)^\top x_0)^2]=\ell(K)^\top\Sigma\ell(K)$.
Writing $\bar\ell:=\ell(\bar K)$, we call the boundary mode \emph{unexcited by $\Sigma$} if $\bar\ell^\top\Sigma\bar\ell=0$. Since $\Sigma\succeq0$, this is equivalent to $\Sigma\bar\ell=0$. In the next section, $y(K)$ and $\ell(K)^\top\Sigma\ell(K)$ describe the decay rate and excitation intensity of the critical mode, respectively.

\section{Finite-time boundary collisions under rank-deficient excitation}
\label{sec:collisions}

We first isolate the contribution of the critical mode to the cost and then derive conditions under which actual gradient trajectories collide with the boundary. An explicit planar system verifies these conditions and provides the basis for constructions in every dimension and at every covariance rank $1\le s<n$.

\subsection{The cost contribution of the critical mode}

Consider the boundary point with a simple zero mode introduced in section~\ref{sec:local-spectrum}. The nearby critical eigenvalue is $-y(K)$, with normalized right and left eigenvectors $v(K)$ and $\ell(K)$. On the stable side, $y(K)$ is the spectral stability margin.

\begin{proposition}[Critical-mode cost decomposition]
\label{prop:pole}
In a sufficiently small neighborhood of the boundary point, there is a function $h_\Sigma$ that extends smoothly across $y=0$ such that
\begin{equation}
 J_\Sigma(K)=h_\Sigma(K)+q(K)\frac{\ell(K)^\top\Sigma\ell(K)}{y(K)},
 ~~ y(K)>0,
 \label{eq:pole}
\end{equation}
where $q(K):=\tfrac12v(K)^\top W_Kv(K)>0$ and $h_\Sigma$ is linear in $\Sigma$.
\end{proposition}

\begin{proof}
With $\Pi(K)=v(K)\ell(K)^\top$, the closed-loop state transition matrix decomposes as
\[
 e^{F_Kt}=e^{-y(K)t}\Pi(K)+E_K(t),
 ~~ E_K(t):=e^{F_Kt}(I-\Pi(K)).
\]
The remaining modes are uniformly strictly stable locally, so $E_K(t)$ and its parameter derivatives of any fixed order satisfy uniform exponential decay bounds. Substituting this decomposition into the cost integral, the critical mode contributes
\[
 \int_0^\infty e^{-2y(K)t}
 \operatorname{tr}\!\left(W_K\Pi(K)\Sigma\Pi(K)^\top\right)\,\mathrm dt
 =q(K)\frac{\ell(K)^\top\Sigma\ell(K)}{y(K)}.
\]
Each remaining term contains at least one factor $E_K(t)$. After shrinking the neighborhood so that $|y(K)|$ is less than half its decay rate, these terms and their parameter derivatives are integrable on both sides of $y=0$. Their integrals define $h_\Sigma$, with the asserted smoothness and linearity.
\end{proof}

Let $\operatorname{rank}\Sigma=s$ and choose a full-column-rank factorization $\Sigma=UU^\top$, with $U\in\mathbb R^{n\times s}$. Define $a(K):=U^\top\ell(K)\in\mathbb R^s$. This vector describes the excitation amplitude of the critical mode, while $\|a(K)\|_2^2$ is its excitation intensity. Equation~\eqref{eq:pole} becomes
\begin{equation}
 J_\Sigma(K)=h_\Sigma(K)+q(K)\frac{\|a(K)\|_2^2}{y(K)}.
 \label{eq:rank-pole}
\end{equation}
If $\|a(K)\|_2$ remains bounded away from zero while $y(K)\to0$, the cost diverges. If instead a path satisfies $a(K)=y(K)r(K)$ with bounded $r(K)$, the critical-mode contribution is $q(K)y(K)\|r(K)\|_2^2\to0$, so the cost can remain bounded. We next determine whether the gradient flow can preserve this relation while driving $y$ to zero in finite optimization time.

\subsection{Direction attraction and finite-time collision}

Two properties must hold simultaneously: the trajectory must remain in a region where $a=O(y)$, and its stability margin must decrease at a strictly negative rate near the boundary. The ratio $r=a/y$ describes the first property; the margin velocity describes the second.

Suppose the boundary point satisfies
\begin{equation}
 U^\top\bar\ell=0,~~ \bar\ell^\top B\ne0,
 ~~ \bar\ell:=\ell(\bar K).
 \label{eq:unexcited-controllable}
\end{equation}
These conditions mean that the initial covariance does not excite the critical mode and that feedback variations can change its eigenvalue. Under these conditions, the map $K\mapsto(a(K),y(K))$ has rank $s+1$ at $\bar K$. It can therefore be completed by tangential variables $z$ to a smooth local coordinate system
\begin{equation}
 \chi(K)=(a,y,z),~~ K=\kappa(a,y,z),~~
 z\in\mathbb R^{mn-s-1}.
 \label{eq:collision-chart}
\end{equation}
Here $\kappa=\chi^{-1}$, and $z$ parameterizes nearby unexcited boundary points satisfying $a=y=0$. The rank calculation is given in appendix~\ref{app:collision-chart}.

For bounded $r=a/y$, the excitation amplitude is $O(y)$ and the intensity is $O(y^2)$. To retain the metric of the original gradient flow, identify the gain space isometrically with $\mathbb R^N$ under the Frobenius inner product, where $N=mn$. At $K_b(z):=\kappa(0,0,z)$, define
\[
 \begin{gathered}
 L:=D_Ka(K_b(z))\in\mathbb R^{s\times N},~~
 d:=\nabla_Ky(K_b(z))\in\mathbb R^N,~~
 g:=\nabla_Kh_\Sigma(K_b(z))\in\mathbb R^N.
 \end{gathered}
\]
The matrix $D_Ka$ is the Jacobian of $a$ with respect to the gain coordinates. We suppress the dependence of these coefficients and $q=q(K_b(z))$ on $z$. On bounded $r$-regions, \eqref{eq:rank-pole} gives
\begin{equation}
 \nabla_KJ_\Sigma(K)=p(r,z)+O(y),~~
 p(r,z):=g+2qL^\top r-q\|r\|_2^2d.
 \label{eq:boundary-gradient}
\end{equation}
Define the margin velocity and direction field by
\begin{equation*}
 G(r,z):=-d^\top p(r,z),~~
 \mathcal F(r,z):=-Lp(r,z)-rG(r,z).
\end{equation*}
Along the original gradient flow,
\[
 y'=G(r,z)+O(y),~~
 r'=\frac{\mathcal F(r,z)}{y}+O(1).
\]
Thus $G<0$ describes motion toward the stability boundary, while $\mathcal F$ determines whether the ratio $r$ remains bounded.

Introducing $\sigma$ through $\mathrm d\sigma/\mathrm d\tau=1/y$ removes the factor $1/y$ from the direction equation:
\begin{equation}
 \begin{aligned}
 \frac{\mathrm dr}{\mathrm d\sigma}=\mathcal F(r,z)+O(y),~
 \frac{\mathrm dy}{\mathrm d\sigma}=yG(r,z)+O(y^2),~
 \frac{\mathrm dz}{\mathrm d\sigma}=O(y),~
 \frac{\mathrm d\tau}{\mathrm d\sigma}=y.
 \end{aligned}
 \label{eq:desingularized}
\end{equation}
This transformation changes only the time parameter, not the gain trajectory. The direction field also has the gradient representation $\mathcal F(r,z)=-\nabla_r\Psi_z(r)$, where
\begin{equation}
 \Psi_z(r)=
 g^\top\!\left(L^\top r-\frac{\|r\|_2^2}{2}d\right)
 +q\left\|L^\top r-\frac{\|r\|_2^2}{2}d\right\|_2^2.
 \label{eq:quartic-potential}
\end{equation}
This is a quartic polynomial in $r$. A nondegenerate local minimum provides an attracting direction; if $G<0$ there, it yields the following collision criterion.

\begin{theorem}[Finite-time collision criterion]
\label{thm:collision-criterion}
Suppose that some $(r_*,z_*)$ satisfies
\begin{equation}
 \nabla_r\Psi_{z_*}(r_*)=0,~~
 \nabla_r^2\Psi_{z_*}(r_*)\succ0,~~ G(r_*,z_*)<0.
 \label{eq:collision-criterion}
\end{equation}
Then a nonempty open set of nearby stabilizing initial gains generates gradient trajectories that reach a boundary point with a simple zero mode in finite optimization time. The strict criterion and the collision conclusion persist under sufficiently small smooth perturbations of the local coefficients and coordinates that preserve the structure of \eqref{eq:desingularized}.
\end{theorem}

\begin{proof}
The first two conditions and the implicit function theorem give a nearby branch of attracting equilibria $r_*(z)$. Together with the third condition, they yield a local Lyapunov function for \eqref{eq:desingularized} and a nonempty open initial set on which
\[
 0<y(\sigma)\le Ce^{-\eta\sigma},~~
 r(\sigma)-r_*(z(\sigma))\to0,
 ~~ \left\|\frac{\mathrm dz}{\mathrm d\sigma}\right\|_2\le Cy(\sigma)
\]
for some $C,\eta>0$; appendix~\ref{app:collision-trap} provides the trapping estimates. Hence $z(\sigma)\to z_\infty$, and the original optimization time has the finite endpoint $T=\int_0^\infty y(\sigma)\,\mathrm d\sigma<\infty$. The gain converges to $K_\infty:=\kappa(0,0,z_\infty)\in\partial\mathcal S$. Direction attraction and the negative margin velocity follow from strict inequalities and therefore persist under the stated perturbations.
\end{proof}

\subsection{An explicit planar system}

Take
\begin{equation}
 A=\begin{pmatrix}-1&5\sqrt3\\0&\sqrt{14}\end{pmatrix},~~
 B=\begin{pmatrix}0\\1\end{pmatrix},~~
 Q=I_2,~ R=1,~
 \Sigma=\begin{pmatrix}1&0\\0&0\end{pmatrix}.
 \label{eq:planar-example}
\end{equation}
The system is controllable and the covariance has rank one. At the boundary gain $\bar K=(0,\sqrt{14})$,
\[
 F_{\bar K}=\begin{pmatrix}-1&5\sqrt3\\0&0\end{pmatrix},~~
 \bar v=\frac1{\sqrt{76}}\begin{pmatrix}5\sqrt3\\1\end{pmatrix},~~
 \bar\ell=\sqrt{76}\begin{pmatrix}0\\1\end{pmatrix}.
\]
With $U=(1,0)^\top$, we have $U^\top\bar\ell=0$ and $\bar\ell^\top B\ne0$. Computation in the original gain coordinates gives
\[
 L=(-\sqrt{76},0),~~
 d=\begin{pmatrix}5\sqrt3\\1\end{pmatrix},~~
 g=\begin{pmatrix}-5\sqrt3/2\\0\end{pmatrix},~~ q=\frac{45}{76}.
\]
The direction variable $r$ is now scalar. Substitution into \eqref{eq:quartic-potential} at $r_*=-\sqrt{57}/6$ yields
\begin{equation}
 \Psi'(r_*)=0,~~ \Psi''(r_*)=\frac{15}{4}>0,
 ~~ G(r_*)=-\frac{15}{4}<0,
 \label{eq:planar-direction}
\end{equation}
where the derivatives of $\Psi$ are taken with respect to $r$. The calculation is detailed in appendix~\ref{app:collision-construction}. Theorem~\ref{thm:collision-criterion} therefore gives an open basin of finite-time collisions.

Along these trajectories, $a/y\to r_*\ne0$, so the modal excitation intensity has order $y^2$ and its cost contribution tends to zero. Meanwhile $y'\to-15/4$, and hence $y(\tau)\sim\tfrac{15}{4}(T-\tau)$. Figure~\ref{fig:mechanism} connects direction attraction with actual gain-space trajectories and illustrates the effect of the weak excitation analyzed in the next section.

\begin{figure}[!htbp]
 \centering
 \includegraphics[width=0.85\textwidth]{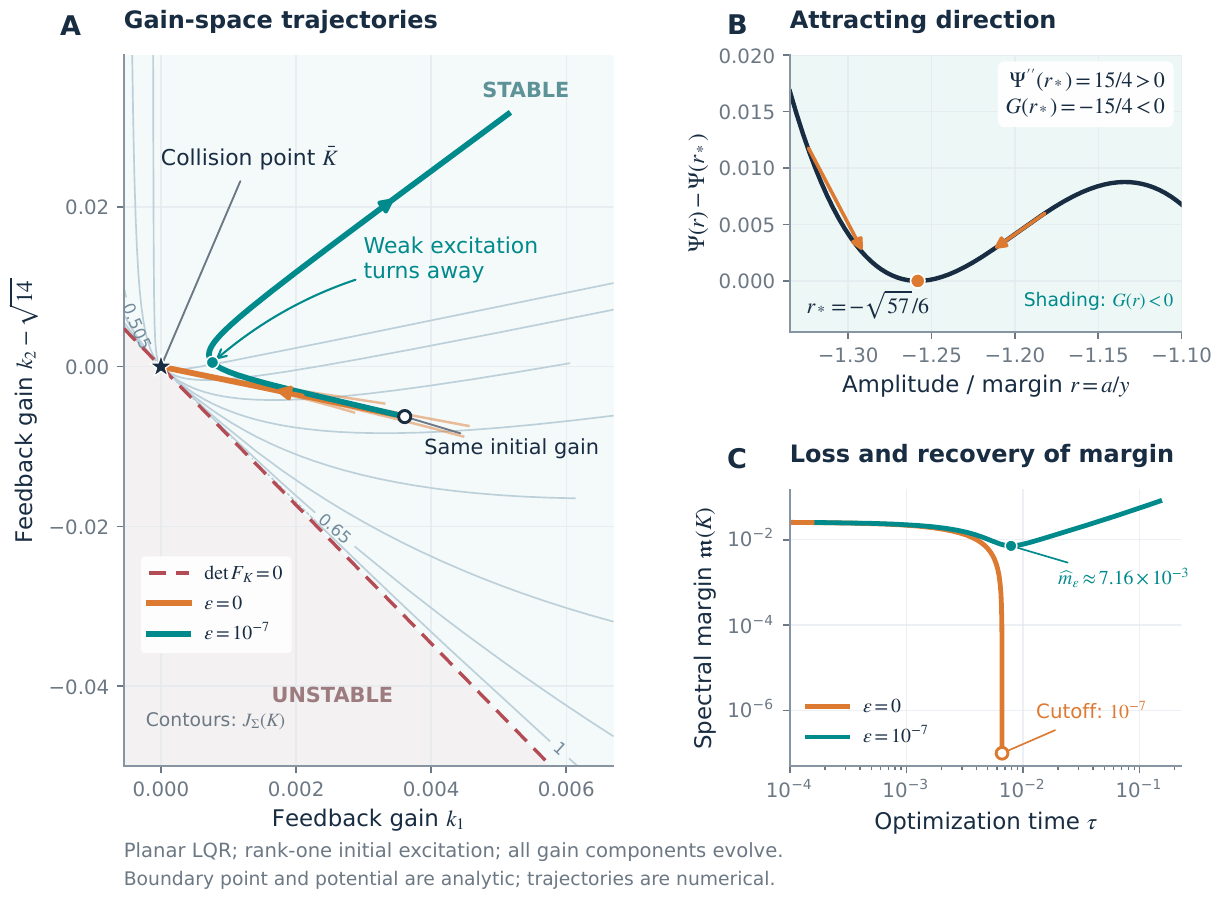}
 \caption{Boundary collision and weak excitation in planar LQR. A: Numerical trajectories in the gain plane $K=(k_1,k_2)$. The dashed line is the local stability boundary and light curves are contours of $J_\Sigma$. Orange trajectories approach the analytic boundary point from different stable initial gains. The teal trajectory starts at the same marked gain and turns away after adding $\varepsilon e_2e_2^\top$, with $\varepsilon=10^{-7}$. B: The direction potential has a strict local minimum at $r_*=-\sqrt{57}/6$, with $G(r_*)<0$. C: Stability margins in the original optimization time. Without added excitation, integration stops at margin $10^{-7}$; the observed minimum $\widehat m_\varepsilon$ after adding excitation is approximately $7.16\times10^{-3}$. The boundary point and potential are analytic; trajectories are obtained by numerical integration.}
 \label{fig:mechanism}
\end{figure}

\subsection{Open collision families at every nonzero deficient rank}

Embedding the planar example in higher dimensions requires the additional directions to remain attracting, while ensuring system controllability and positive definiteness of the state Gramian before collision. The following construction achieves all three properties.

\begin{theorem}[Open collision families at every nonzero deficient rank]
\label{thm:open-collision}
For every $n\ge2$, $m\ge1$, and $1\le s<n$, there exists a nonempty relatively open set of data $(A,B,Q,R,\Sigma)$ satisfying
\[
 (A,B)\text{ controllable},~~ Q\succ0,~ R\succ0,
 ~~ \Sigma\succeq0,~\operatorname{rank}\Sigma=s,
\]
such that each data set admits a nonempty open set of stabilizing initial gains whose gradient trajectories reach a boundary point with a simple zero mode in finite optimization time. The initial set can also be chosen so that $X_{K(\tau),\Sigma}\succ0$ at every time before collision. Relative openness is taken in the data space with covariance rank fixed at $s$.
\end{theorem}

\begin{proof}[Proof outline]
Append distinct stable modes $-\lambda_3,\ldots,-\lambda_n$ to the planar system, with $2<\lambda_j<3$, and set
\[
 A=A_2\oplus\operatorname{diag}(-\lambda_3,\ldots,-\lambda_n),
 ~~ Q=I_n,~ R=I_m,
\]
where $A_2$ is the system matrix in \eqref{eq:planar-example}. Let $e_j$ denote the $j$th standard basis vector of $\mathbb R^n$. Set the first input column to $e_2$ and the remaining columns to zero, and choose $U=(e_1,e_3,\ldots,e_{s+1})$, $\Sigma=UU^\top$. The direction linearization in each additional excited direction has eigenvalue
\begin{equation}
 -\frac{90}{\lambda_j^2}+\frac{15}{4}< -\frac{25}{4}<0.
 \label{eq:extra-direction-eigen}
\end{equation}
Thus the additional directions remain attracting, and the margin velocity of the original planar direction remains $-15/4$. The other gain components are tangential variables $z$.

Next, make sufficiently small perturbations
\[
 B_{:,1}=e_2+\eta\sum_{j=3}^n e_j,
 ~~ U_1=e_1+\delta\sum_{j=s+2}^n e_j,
 ~~ \eta,\delta>0,
\]
leaving the remaining columns of $U$ unchanged. The first perturbation makes the system controllable; the second excites every stable mode at the boundary. The covariance rank remains $s$, and strict direction attraction and the negative margin velocity persist. In a sufficiently small collision basin, $a=yr\ne0$ also holds, so the critical mode is excited. The controllability characterization of the Gramian then gives $X_{K,\Sigma}\succ0$.

Finally, the full-rank coordinate map allows the unexcited boundary point to move with the data. Controllability, positive definiteness of the weights, and the strict collision criterion all persist under small perturbations. This yields the relatively open data set and the open initial set. Appendix~\ref{app:collision-construction} provides the full construction and the calculation for the additional directions.
\end{proof}

For these families, let $K_\infty$ be the collision endpoint, and set
\[
 r_\infty:=\lim_{\tau\uparrow T}\frac{a(K(\tau))}{y(K(\tau))}\ne0,
 ~~ \ell_\infty:=\ell(K_\infty).
\]
The smallest eigenvalue of the state Gramian satisfies
\begin{equation}
 \lambda_{\min}(X_{K(\tau),\Sigma})
 =\frac{\|r_\infty\|_2^2}{2\|\ell_\infty\|_2^2}
 y(K(\tau))+o\!\left(y(K(\tau))\right).
 \label{eq:gramian-asymptotic}
\end{equation}
The proof is given in appendix~\ref{app:gramian-rate}. Thus the Gramian remains positive definite before collision, but its smallest eigenvalue vanishes linearly with the stability margin. Pointwise positive definiteness does not provide a uniform excitation bound along the optimization trajectory.

\section{Stability margins under weak excitation}\label{sec:excitation}

Section~\ref{sec:collisions} shows that excitation of the critical mode can degenerate along a gradient trajectory, leading to a boundary collision. We now add excitation in the missing directions. We first establish a local barrier against collision, then derive matching bounds on the minimum stability margin, and finally extend these bounds to the entire trajectory under a single-input, corank-one structure.

Write the original covariance as $\Sigma_0$ and set
\begin{equation*}
 \Sigma_\varepsilon:=\Sigma_0+\varepsilon H,
 ~~ H\succeq0,~ \Sigma_0+H\succ0,
\end{equation*}
where $H$ is fixed and $\varepsilon>0$ is the added excitation strength. Let $K_\varepsilon$ denote the corresponding gradient trajectory from a fixed gain $K_0\in\mathcal S$, independent of $\varepsilon$.

\subsection{A local barrier induced by added excitation}

Consider an unexcited simple zero mode as in Section~\ref{sec:collisions}. Linearity of the cost in the covariance and the pole decomposition give
\begin{equation}
 J_{\Sigma_\varepsilon}(K)
 =h_{\Sigma_0}(K)
  +q(K)\frac{\|a(K)\|^2}{y(K)}
  +\varepsilon\frac{b(K)}{y(K)}
  +\varepsilon h_H(K),
 \label{eq:regularized-pole}
\end{equation}
where $b(K):=q(K)\ell(K)^\top H\ell(K)$. At the original collision point $\bar K$, we have $\Sigma_0\bar\ell=0$. Since $\Sigma_0+H\succ0$, it follows that $b(\bar K)>0$, so $b$ has a strictly positive lower bound in a sufficiently small neighborhood.

The added term $\varepsilon b/y$ does not vanish as $a\to0$, and its gradient contains a term of order $\varepsilon/y^2$. We estimate its effect on the stability margin directly in the original gain space equipped with the Frobenius metric.

\begin{proposition}[Local stability-margin barrier]\label{prop:local-barrier}
Suppose that $F_{\bar K}$ has a simple zero eigenvalue and all its other eigenvalues lie strictly in the left half-plane. Assume also that $\bar\ell^\top(\Sigma_0+H)\bar\ell>0$ and $\bar\ell^\top B\ne0$. Then there exist a neighborhood $V$ of $\bar K$ and constants $c_0,C_0,\varepsilon_0>0$ such that, for every $0<\varepsilon\le\varepsilon_0$, the gradient flow satisfies
\begin{equation}
 \frac{\mathrm d}{\mathrm d\tau}y(K_\varepsilon(\tau))
 \ge \frac{c_0\varepsilon}{y(K_\varepsilon(\tau))^2}-C_0,
 ~~ K_\varepsilon(\tau)\in V\cap\mathcal S.
 \label{eq:local-barrier}
\end{equation}
The constants are independent of $\varepsilon$, and $V$ can be chosen so that $y=\mathfrak m$ on its stable side.
\end{proposition}

\begin{proof}
Set $\rho_\varepsilon(K):=q(K)\ell(K)^\top\Sigma_\varepsilon\ell(K)$ and $h_\varepsilon:=h_{\Sigma_\varepsilon}$. The pole decomposition yields
\begin{equation*}
 y'
 =\frac{\rho_\varepsilon\|\nabla y\|_F^2}{y^2}
  -\frac{\langle\nabla y,\nabla\rho_\varepsilon\rangle_F}{y}
  -\langle\nabla y,\nabla h_\varepsilon\rangle_F.
\end{equation*}
Shrink $V$ so that its closure is a compact subset of the simple-mode coordinate neighborhood, and take $0<\varepsilon\le1$. Since
\[
 \Sigma_\varepsilon\succeq\varepsilon(\Sigma_0+H),
 ~~
 D_Ky(\bar K)[\Delta K]=\bar\ell^\top B\,\Delta K\bar v,
\]
there are positive constants $b_0,d_0$ such that $\rho_\varepsilon\ge b_0\varepsilon$ and $\|\nabla y\|_F^2\ge d_0$. Writing $\rho_\varepsilon=q\|\Sigma_\varepsilon^{1/2}\ell\|^2$, we use the positive lower bound on $q$, the smoothness of $q$ and $\ell$, and $0\preceq\Sigma_\varepsilon\preceq\Sigma_0+H$ to obtain
\[
 \|\nabla\rho_\varepsilon\|_F\le L_0\sqrt{\rho_\varepsilon},
\]
with $L_0$ independent of $\varepsilon$. The gradient of $h_\varepsilon$ is also uniformly bounded. Young's inequality therefore gives
\[
 y'\ge d_0\frac{\rho_\varepsilon}{y^2}
       -L_1\frac{\sqrt{\rho_\varepsilon}}y-C_1
 \ge\frac{d_0}{2}\frac{\rho_\varepsilon}{y^2}-C_0
 \ge\frac{d_0b_0}{2}\frac{\varepsilon}{y^2}-C_0,
\]
which proves \eqref{eq:local-barrier}.
\end{proof}

Let $c_*:=\sqrt{c_0/(2C_0)}$. At $y=c_*\sqrt\varepsilon$, \eqref{eq:local-barrier} gives $y'\ge C_0>0$. Thus a trajectory starting above this level cannot cross it downward while remaining in $V$.

The square-root scale results from balancing the $\varepsilon/y^2$ term against a bounded negative margin velocity. To show that this scale is sharp, we must also prove that perturbations of colliding trajectories reach a region where $y=O(\sqrt\varepsilon)$.

\subsection{Matching square-root bounds near colliding trajectories}

Use the coordinates $(a,y,z)$, the ratio $r=a/y$, and the direction field $\mathcal F$ from Section~\ref{sec:collisions}. Suppose that a boundary point admits a direction satisfying the collision criterion \eqref{eq:collision-criterion}: the direction is strictly attracting, and the limiting margin velocity is negative.

\begin{theorem}[Local minimum stability margin]\label{thm:local-sqrt}
There exist a relatively compact neighborhood $V$ of this boundary point and a nonempty open set of local initial gains $\mathcal C\subset V\cap\mathcal S$ such that, for every fixed $K_0\in\mathcal C$, there are constants $c,C,\varepsilon_0>0$ for which
\begin{equation}
 c\sqrt\varepsilon
 \le\inf_{0\le\tau<\tau_{\mathrm{exit}}^\varepsilon}
       \mathfrak m(K_\varepsilon(\tau))
 \le C\sqrt\varepsilon,
 ~~ 0<\varepsilon\le\varepsilon_0,
 \label{eq:local-sqrt}
\end{equation}
where $\tau_{\mathrm{exit}}^\varepsilon:=\inf\{\tau\ge0:K_\varepsilon(\tau)\notin V\}$, with $\inf\varnothing:=+\infty$. The constants may depend on the problem data, $H$, and $K_0$, but not on $\varepsilon$. For every compact set $\mathcal C_0\subset\mathcal C$, they can be chosen uniformly over $K_0\in\mathcal C_0$.
\end{theorem}

\begin{proof}
The lower bound follows from Proposition~\ref{prop:local-barrier}. A fixed initial gain has a positive stability margin, which lies above the local barrier for sufficiently small $\varepsilon$.

For the upper bound, choose a small neighborhood of the strictly attracting direction such that the unperturbed direction field points into this neighborhood at its boundary and $G(r,z)\le-2\gamma$ for some $\gamma>0$. By \eqref{eq:regularized-pole}, on a region with bounded ratio variable the perturbed equations satisfy
\begin{equation*}
 \frac{\mathrm dr}{\mathrm d\sigma}
  =\mathcal F(r,z)+O(y)+O(\varepsilon/y^2),
 ~~
 y'=G(r,z)+O(y)+O(\varepsilon/y^2),
\end{equation*}
where $\mathrm d\sigma/\mathrm d\tau=1/y$ and the remainder estimates are uniform on the chosen local region. First shrink the coordinate neighborhood and then choose a sufficiently large fixed constant $L_*$. As long as $y\ge L_*\sqrt\varepsilon$, the direction field still points into the direction neighborhood at its boundary, and $y'\le-\gamma$. During the same stage,
\[
 \frac{\mathrm dz}{\mathrm d\sigma}
 =O(y)+O(\varepsilon/y)=O(y),
\]
because $\varepsilon/y\le y/L_*^2$. Hence $\|z'\|$ is uniformly bounded, while the original optimization time needed to reach the lower level is at most $y(K_0)/\gamma$. Choosing the initial height sufficiently small ensures that the trajectory reaches $y=L_*\sqrt\varepsilon$ before leaving $V$. This proves the upper bound.

The direction and drift estimates hold uniformly on a fixed local region. On a compact subset of $\mathcal C$, the initial margin also has a common positive lower bound, so $\varepsilon_0$ can be chosen uniformly. Appendix~\ref{app:local-passage} gives the detailed hitting estimates.
\end{proof}

A bounded average cost does not provide uniform control of performance in every initial-state direction. For deterministic initial states of unit norm, the largest accumulated cost is $\sup_{\|x_0\|_2=1}x_0^\top P_Kx_0=\lambda_{\max}(P_K)$. The local square-root bounds determine the order of this worst-case cost.

\begin{corollary}[Worst-case initial-state cost]\label{cor:worst-cost}
Under the assumptions of Theorem~\ref{thm:local-sqrt}, for the same fixed initial gain there are constants $c_P,C_P>0$ such that, for all sufficiently small $\varepsilon>0$,
\begin{equation}
 c_P\varepsilon^{-1/2}
 \le\sup_{0\le\tau<\tau_{\mathrm{exit}}^\varepsilon}
       \lambda_{\max}(P_{K_\varepsilon(\tau)})
 \le C_P\varepsilon^{-1/2}.
 \label{eq:worst-cost}
\end{equation}
\end{corollary}

\begin{proof}
The spectral decomposition of the critical mode gives
\[
 P_K=\frac{q(K)}{y(K)}\ell(K)\ell(K)^\top+\mathcal R_P(K),
\]
where $\mathcal R_P$ extends smoothly across $y=0$ and is bounded on $\overline V$. Thus $\lambda_{\max}(P_K)$ is comparable to $1/y(K)$ sufficiently close to the boundary. Combining \eqref{eq:local-sqrt} with the hitting point in its proof gives \eqref{eq:worst-cost}. The constant bounds are given in Appendix~\ref{app:worst-cost}.
\end{proof}

\subsection{Localization of low-margin gains under a cost bound}

The local barrier controls motion near a boundary point. A lower bound for the entire trajectory must also rule out smaller margins elsewhere. The following single-input, corank-one structure localizes all low-margin gains with bounded cost near a single boundary point.

Let $m=1$ and $\operatorname{rank}\Sigma_0=n-1$, and take a full-column-rank factorization $\Sigma_0=UU^\top$, with $U\in\mathbb R^{n\times(n-1)}$. Write $\mathcal X_0:=\operatorname{range}\Sigma_0$ for the initial excitation subspace, and let the unit vector $\nu$ span $\ker\Sigma_0$, so that $\mathcal X_0=\nu^\perp$. Assume that
\begin{equation}
 \nu^\top B\ne0,~~
 \bar K:=\frac{\nu^\top A}{\nu^\top B},~~
 Z_0:=(A-B\bar K)|_{\mathcal X_0}\text{ is Hurwitz}.
 \label{eq:global-structure}
\end{equation}
Since $\nu^\top(A-B\bar K)=0$, the subspace $\mathcal X_0$ is invariant under this closed-loop system, and $Z_0$ is the restriction of its dynamics to that subspace. Condition~\eqref{eq:global-structure} ensures that $F_{\bar K}$ has a simple zero mode unexcited by $\Sigma_0$, while all its other modes are stable.

\begin{lemma}[Low-margin localization under a cost bound]\label{lem:localization}
Under \eqref{eq:global-structure}, for every finite cost bound $M>0$ and every open neighborhood $V$ of $\bar K$, there exists $\eta>0$ such that
\begin{equation}
 K\in\mathcal S,~ J_{\Sigma_0}(K)\le M,~
 \mathfrak m(K)\le\eta
 ~\Longrightarrow~ K\in V.
 \label{eq:localization}
\end{equation}
\end{lemma}

\begin{proof}
We first exclude unbounded gains. The Lyapunov equation yields the control-energy estimate
\begin{equation}
 R\|KU\|_F^2\le-2\operatorname{tr}(F_K)J_{\Sigma_0}(K),
 \label{eq:control-energy-bound}
\end{equation}
where $R>0$ is a scalar. The condition $\nu^\top B\ne0$ makes $T=[U,B]$ invertible. In these state coordinates, write
\[
 T^{-1}AT=\begin{pmatrix}Z&\omega\\ \beta&\alpha\end{pmatrix},
 ~~ KT=(u,v)=(KU,KB),
\]
where $Z\in\mathbb R^{(n-1)\times(n-1)}$ represents $Z_0$, $\omega$ is a column vector, $\beta$ is a row vector, and $\alpha,v$ are scalars. The closed-loop matrix is similar to
\begin{equation}
 T^{-1}F_KT
 =\begin{pmatrix}Z&\omega\\ \beta-u&\alpha-v\end{pmatrix}.
 \label{eq:adapted-closed-loop}
\end{equation}
Since $-\operatorname{tr}F_K=v-\operatorname{tr}A>0$, \eqref{eq:control-energy-bound} and $J_{\Sigma_0}\le M$ give $\|u\|_2^2\le(2M/R)(v-\operatorname{tr}A)$. Consequently, if $\|K_j\|_F\to\infty$, then $v_j\to+\infty$ and $\|u_j\|_2=O(\sqrt{v_j})$. Thus the gain components acting on the initial excitation subspace grow more slowly than $v_j$. Dividing the characteristic polynomial of \eqref{eq:adapted-closed-loop} by $v_j$ shows that $n-1$ eigenvalues converge to the spectrum of $Z$, while the remaining eigenvalue tends to $-\infty$. Therefore,
\[
 \mathfrak m(K_j)\longrightarrow
 -\max_{\lambda\in\operatorname{spec}(Z_0)}\operatorname{Re}\lambda>0.
\]
Appendix~\ref{app:high-gain} derives the energy estimate and spectral limit in detail. The state-coordinate change here is used only for spectral analysis; the gradient flow remains defined by the original Frobenius metric.

It follows that every bounded-cost gain sequence whose stability margin tends to zero is bounded. Let $\widehat K$ be any of its accumulation points. Fatou's lemma and $Q\succ0$ imply
\[
 \int_0^\infty\|e^{F_{\widehat K}t}U\|_F^2\,\mathrm dt
 \le\frac{M}{\lambda_{\min}(Q)}<\infty.
\]
Hence $\mathcal X_0$ is contained in the strictly stable generalized eigenspace of $F_{\widehat K}$. Since $\dim\mathcal X_0=n-1$ and $\widehat K$ lies on the stability boundary, this stable subspace equals $\mathcal X_0$, and the remaining mode must be a simple zero mode. Invariance of the subspace and the zero eigenvalue on the quotient space give $\nu^\top(A-B\widehat K)=0$. As $\nu^\top B\ne0$, we obtain $\widehat K=\bar K$. Every accumulation point is therefore the same, so every bounded-cost sequence with vanishing margin converges to $\bar K$. This proves \eqref{eq:localization}.
\end{proof}

\subsection{The square-root law over the entire trajectory}

\begin{theorem}[Global stability margin in the single-input, corank-one case]\label{thm:global-sqrt}
Suppose that \eqref{eq:global-structure} holds, fix any $K_0\in\mathcal S$, and take $H\succeq0$ with $\nu^\top H\nu>0$. For every $\varepsilon>0$, the gradient flow of $J_{\Sigma_0+\varepsilon H}$ from $K_0$ exists globally and converges to the Riccati optimal gain. Moreover, there are constants $c,\varepsilon_0>0$, independent of $\varepsilon$, such that
\begin{equation*}
 \inf_{\tau\ge0}\mathfrak m(K_\varepsilon(\tau))
 \ge c\sqrt\varepsilon,
 ~~ 0<\varepsilon\le\varepsilon_0.
\end{equation*}
If the unperturbed trajectory from the same initial gain collides at $\bar K$ and its limiting direction satisfies \eqref{eq:collision-criterion}, then there also exists $C>0$ such that
\begin{equation}
 c\sqrt\varepsilon
 \le\inf_{\tau\ge0}\mathfrak m(K_\varepsilon(\tau))
 \le C\sqrt\varepsilon,
 ~~ 0<\varepsilon\le\varepsilon_0.
 \label{eq:global-sqrt}
\end{equation}
\end{theorem}

\begin{proof}
The condition $\nu^\top H\nu>0$ ensures that $\Sigma_0+\varepsilon H\succ0$, so global existence and convergence for each fixed $\varepsilon>0$ follow from Proposition~\ref{prop:full-rank}.

For $0<\varepsilon\le1$, cost monotonicity gives the common bound
\begin{equation}
 \begin{aligned}
 J_{\Sigma_0}(K_\varepsilon(\tau))
 &\le J_{\Sigma_\varepsilon}(K_\varepsilon(\tau))\\
 &\le J_{\Sigma_\varepsilon}(K_0)
 \le J_{\Sigma_0+H}(K_0)=:M.
 \end{aligned}
 \label{eq:common-cost-bound}
\end{equation}
At $\bar K$, the left zero eigenvector is proportional to $\nu$, so Proposition~\ref{prop:local-barrier} applies. Choose its neighborhood $V$ so that $y=\mathfrak m$ on $V\cap\mathcal S$, and let $\eta>0$ be the threshold supplied by Lemma~\ref{lem:localization} for $M,V$.

Take $c=\sqrt{c_0/(2C_0)}$ and decrease $\varepsilon_0$ so that $c\sqrt{\varepsilon_0}<\min\{\eta,\mathfrak m(K_0)\}$. If the trajectory first crossed the margin level $c\sqrt\varepsilon$ downward, \eqref{eq:common-cost-bound} and low-margin localization would place the crossing point in $V$. There $\mathfrak m=y$ is smooth, and \eqref{eq:local-barrier} gives $y'\ge C_0>0$, contradicting a first downward crossing. The lower bound therefore holds for all optimization times.

For the upper bound, the unperturbed colliding trajectory enters the local initial set $\mathcal C$ of Theorem~\ref{thm:local-sqrt} at some fixed time $\tau_0$ before collision. Choose a small neighborhood of this entry point whose closure is a compact subset of $\mathcal C$. Continuous dependence on parameters on $[0,\tau_0]$ ensures that $K_\varepsilon(\tau_0)$ lies in this neighborhood for all sufficiently small $\varepsilon$. The hitting estimates in Theorem~\ref{thm:local-sqrt} are uniform there, so each such trajectory subsequently reaches $y\le C\sqrt\varepsilon$. This proves the upper bound in \eqref{eq:global-sqrt}.
\end{proof}

The choice $H=\nu\nu^\top$ satisfies the excitation condition, so rank-one added excitation suffices. The single-input, corank-one collision families constructed in Theorem~\ref{thm:open-collision} satisfy \eqref{eq:global-structure}: their unexcited left zero eigenvector spans $\ker\Sigma_0$, controllability gives $\nu^\top B\ne0$, and the remaining stable modes span $\mathcal X_0$. They therefore obey the square-root law over the entire trajectory. The proof uses a common cost bound and low-margin localization, without requiring a common compact set for trajectories with different $\varepsilon$.

\section{Numerical validation}\label{sec:numerics}

We use two single-input examples with covariance of corank one to examine the collision criterion, Gramian degeneration, and the square-root scale under weak excitation. The three-dimensional example includes coupling perturbations in the system and input matrices to examine the perturbed collision mechanism; the five-dimensional example adds state and excitation directions to examine the scale in higher dimension. Both use $Q=I_n$, $R=1$, and $\Sigma_\varepsilon=\Sigma_0+\varepsilon H$, where $\operatorname{rank}\Sigma_0=n-1$ and $H$ is the rank-one orthogonal projection onto the missing excitation direction. Figure~\ref{fig:experiments} summarizes the results.

\subsection{Examples and computational setup}

\paragraph{A coupled three-dimensional system}
Consider
\[
A=\begin{pmatrix}
-1&5\sqrt3&0.008\\
0.006&\sqrt{14}&-0.004\\
-0.007&0.009&-2.5
\end{pmatrix},~~
B=\begin{pmatrix}0.003\\1\\0.005\end{pmatrix},
\]
with $\Sigma_0=\operatorname{diag}(1,0,1)$ and $H=e_2e_2^\top$. This system is controllable, and its unexcited boundary gain is $\bar K=(0.006,\sqrt{14},-0.004)$. The eigenvalues of the closed-loop restriction to the initial excitation subspace are approximately $-1.000056$ and $-2.499942$, so condition~\eqref{eq:global-structure} holds. The fixed initial gain is
\[
K_0^{(3)}\approx
(0.0096152487,\ 3.7353900118,\ -0.0040128881).
\]

\paragraph{A structured five-dimensional system}
Consider
\[
A=\begin{pmatrix}-\Lambda&c\\0&\sqrt{14}\end{pmatrix},
~~ B=e_5,~~
\Sigma_0=\operatorname{diag}(I_4,0),~~ H=e_5e_5^\top,
\]
where $\Lambda=\operatorname{diag}(1,2.5,3,4)$ and $c=(5\sqrt3,0.03,0.02,0.015)^\top$. The unexcited boundary gain is $\bar K=(0,0,0,0,\sqrt{14})$, and the closed-loop restriction to the initial excitation subspace is $-\Lambda$. We fix
\[
\begin{aligned}
K_0^{(5)}\approx\bigl(0.0036081620,\ 0.0000286601,\ 0.0000264806, \ 
0.0000484457,\ 3.7354090851\bigr).
\end{aligned}
\]
The directional calculations below test the collision criterion for this example.

\paragraph{Integration and minimum-margin computation}
For each example, all runs use the same initial gain at full numerical precision and update every gain component. Near the boundary, we separate the critical-mode pole and reparameterize time. Once a positive-excitation trajectory first reaches margin $0.08$ in the outward direction, we resume integration in the original optimization time until
\begin{equation}
\|K_\varepsilon-K_\star\|_F\le10^{-7}.
\label{eq:numerical-stop}
\end{equation}
At the initial points used in both examples, the normalized discrepancy between the pole-separated gradient and direct Lyapunov evaluation is at most $1.11\times10^{-12}$. Integration methods and tolerances are specified in Appendix~\ref{app:numerical-settings}.

\subsection{Collision criteria and unregularized trajectories}

We solve $\mathcal F(r_*)=0$ numerically and evaluate the margin velocity $G(r_*)$ and direction Jacobian $D_r\mathcal F(r_*)$. In Table~\ref{tab:direction-check}, the negative maximum real part of the Jacobian eigenvalues indicates strict directional attraction, while $G(r_*)<0$ indicates a decreasing stability margin. These values agree with the criterion in Theorem~\ref{thm:collision-criterion}.

\begin{table}[tbp]
\centering
\caption{Direction equilibria and margin velocities.}
\label{tab:direction-check}
\begin{tabular}{lcc}
\toprule
Example & $G(r_*)$ & $\max\operatorname{Re}\operatorname{spec}(D_r\mathcal F(r_*))$\\
\midrule
Coupled 3D & $-3.761512$ & $-3.792327$\\
Structured 5D & $-3.750554$ & $-1.871803$\\
\bottomrule
\end{tabular}
\end{table}

For the three-dimensional system with $\varepsilon=0$, Table~\ref{tab:collision-records} retains three records from the original integration stopped at $y=10^{-6}$. Figure~\ref{fig:experiments}A--C extends the trajectory for the same example to $y=10^{-7}$ to display its behavior closer to the boundary.

\begin{table}[ht]
\centering
\caption{Near-boundary records for the three-dimensional unregularized trajectory, with $X_K=X_{K,\Sigma_0}$.}
\label{tab:collision-records}
\small
\begin{tabular}{cccc}
\toprule
Margin $y$ & $y'$ & $\lambda_{\min}(X_K)/y$ & $J_{\Sigma_0}(K)$\\
\midrule
$1.13930\times10^{-2}$ & $-3.897943$ & $0.01032537$ & $0.703558565$\\
$1.04835\times10^{-4}$ & $-3.762994$ & $0.01045404$ & $0.700044010$\\
$1.00000\times10^{-6}$ & $-3.761529$ & $0.01045540$ & $0.700011300$\\
\bottomrule
\end{tabular}
\end{table}

As the margin decreases, $y'$ approaches the negative constant predicted by the direction analysis, while the cost approaches the finite boundary value $0.700010985$. The state Gramian remains positive definite, and $\lambda_{\min}(X_K)/y$ approaches the positive constant $0.0104554$. Figure~\ref{fig:experiments}B,C shows the bounded cost and approximately linear Gramian degeneration, consistent with the pole decomposition and~\eqref{eq:gramian-asymptotic}.

\subsection{Minimum stability margins under weak excitation}

For $\varepsilon\in\{10^{-4},10^{-5},\ldots,10^{-10}\}$, we integrate each positive-excitation trajectory until~\eqref{eq:numerical-stop}. Let $\widetilde K_\varepsilon$ denote the continuous interpolant supplied by the integrator and $\tau_{\mathrm{end}}^\varepsilon$ its termination time. The numerical minimum over the computed interval is
\[
\widehat m_\varepsilon:=
\min_{0\le\tau\le\tau_{\mathrm{end}}^\varepsilon}
\mathfrak m(\widetilde K_\varepsilon(\tau)).
\]
We search the continuous interpolant and refine each detected local minimum.

\begin{figure}[p]
\centering
\includegraphics[width=0.85\linewidth,height=.80\textheight,keepaspectratio]{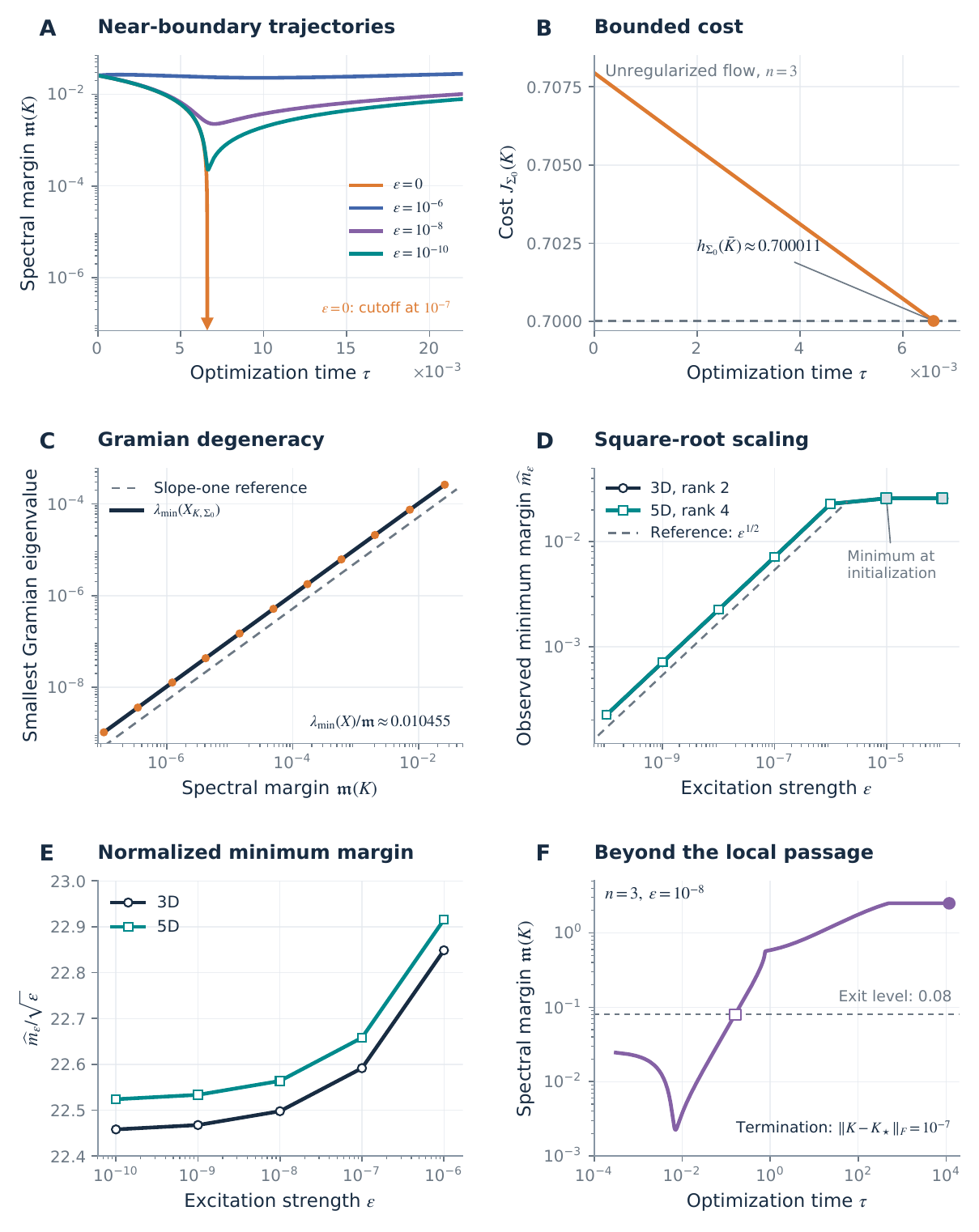}
\caption{Near-boundary behavior under deficient excitation and the weak-excitation scale. A--C: margin, bounded cost, and Gramian degeneration for the coupled three-dimensional system; the unregularized trajectory is stopped at margin $10^{-7}$. D,E: numerical minimum margins for the three- and five-dimensional systems; the dashed line in D is a square-root reference. F: the three-dimensional trajectory at $\varepsilon=10^{-8}$ from its local passage to the Riccati neighborhood; the square marks the first outward crossing of margin $0.08$. The computations in D--F terminate according to~\eqref{eq:numerical-stop}.}
\label{fig:experiments}
\end{figure}

Figure~\ref{fig:experiments}A shows the turn away from the boundary after excitation is added. Panels D and E display the minimum margin and its normalized value $\widehat m_\varepsilon/\sqrt\varepsilon$. Table~\ref{tab:sqrt-scale} reports three representative excitation levels.

\begin{table}[ht]
\centering
\caption{Normalized numerical minimum margins $\widehat m_\varepsilon/\sqrt\varepsilon$ under weak excitation.}
\label{tab:sqrt-scale}
\begin{tabular}{crr}
\toprule
$\varepsilon$ & 3D system & 5D system\\
\midrule
$10^{-6}$ & $22.848704$ & $22.915459$\\
$10^{-8}$ & $22.497852$ & $22.563846$\\
$10^{-10}$ & $22.458138$ & $22.523928$\\
\bottomrule
\end{tabular}
\end{table}

The normalized margins level off as $\varepsilon$ decreases. Reducing the excitation from $10^{-6}$ to $10^{-10}$ reduces the minimum margin by approximately two orders of magnitude, consistent with the predicted $\Theta(\sqrt\varepsilon)$ scale. In every run from $10^{-6}$ to $10^{-10}$, the minimum over the computed interval occurs at the first local minimum. The trajectory then leaves the boundary neighborhood; after its first outward crossing of margin $0.08$, it never falls below that level before termination. Figure~\ref{fig:experiments}F shows the complete record for the three-dimensional system at $\varepsilon=10^{-8}$.

For $\varepsilon=10^{-4}$ or $10^{-5}$, the minimum in both examples occurs at the initial gain and is approximately $0.02583$; these points are marked by pale gray fill in Figure~\ref{fig:experiments}D. Thus, for the fixed initialization, the square-root variation appears in the weaker-excitation regime. Replacing the rank-one excitation by $H=I_n$ yields the same scaling trend.

All reported numerical minima are taken over the finite interval from initialization to the prescribed Riccati neighborhood. The corresponding all-time bounds follow from Theorem~\ref{thm:global-sqrt}.

\section{Conclusion}

We studied how initial-state excitation governs stability along continuous-time LQR policy-gradient flows. For every nonzero deficient covariance rank, we constructed relatively open families of controllable systems whose exact Euclidean gradient trajectories reach the stability boundary in finite optimization time. The state Gramian can remain positive definite before collision, showing that pointwise excitation through the dynamics does not suffice to preserve stability. Adding excitation of strength $\varepsilon$ restores stability and yields matching $\Theta(\sqrt{\varepsilon})$ bounds on the minimum margin during the local boundary passage. These bounds extend to the entire trajectory for the specified single-input, corank-one families. Together, the results identify a mechanism by which cost descent can approach instability and quantify the protection provided by weak excitation.
\FloatBarrier
\bibliographystyle{siamplain}
\bibliography{reference}
\newpage
\appendix
\section{Related Works}
\label{app:relate}
Stability-preserving optimization for LQR originates in Riccati-based policy iteration. The classical algorithms of Kleinman and Hewer generate stabilizing policies from a stabilizing initialization and converge to the optimal regulator~\cite{kleinman1968,hewer1971}. Modern policy-gradient theory studies the same synthesis problem directly in the feedback-gain space. Fazel et al.~\cite{fazel2018} established global convergence guarantees for discrete-time LQR, while subsequent analyses developed the cost geometry and convergence theory for discrete- and continuous-time gradient methods~\cite{bu2019discrete,mohammadi2019,bu2020}. In the continuous-time setting with positive definite initial covariance, compactness of cost sublevel sets within the stabilizing region connects cost descent with stability preservation.

Further developments address how stability is obtained and maintained when implementing policy optimization. Derivative-free analyses quantify the effect of estimating policy updates from cost observations~\cite{malik2020,mohammadi2019}. Discount annealing provides a route to stabilization without a given stabilizing controller~\cite{perdomo2021}, while robustness analyses of policy iteration establish tolerance to small evaluation and update errors~\cite{pang2020continuous,pang2021discrete}. For constrained policy optimization, geometric stability certificates support local search within the stabilizing set~\cite{talebi2022}; the broader role of policy parameterization and geometry is surveyed in~\cite{talebi2024geometry}. These developments distinguish stabilization, stability preservation, and robustness of the learning dynamics.

The assumptions linking the objective to stability are central to the present work. Watanabe and Zheng~\cite{watanabe2026sicon} revisit strong duality, hidden convexity, and gradient dominance in LQR. Their unified analysis~\cite{watanabe2026unified} accommodates positive semidefinite initial covariance under an additional Lyapunov-solution condition and establishes uniform gradient-dominance bounds on compact subsets of stabilizing gains. With rank-deficient excitation, however, pointwise positive definiteness of the state Gramian does not itself ensure that a gradient trajectory remains in such a compact subset.

We examine this issue through the boundary dynamics of the exact Euclidean policy-gradient flow. We construct robust finite-time collisions with the stability boundary even when the state Gramian remains positive definite before collision. We then show that adding excitation of strength $\varepsilon$ produces matching local bounds of order $\sqrt{\varepsilon}$ for the minimum stability margin, with whole-trajectory bounds for the specified single-input, corank-one families. This characterizes both the failure of stability preservation under deficient excitation and the quantitative protection supplied by weak excitation.

\section{Supporting proofs}\label{app:proofs}

\subsection{The Lyapunov gradient identity}\label{app:gradient}

Differentiating~\eqref{eq:setting-P} in a gain direction $\Delta K$ gives
\[
F_K^\top\Delta P+\Delta PF_K
+\Delta K^\top(RK-B^\top P_K)
+(RK-B^\top P_K)^\top\Delta K=0.
\]
Multiplying on the right by $X_{K,\Sigma}$, taking traces, and using~\eqref{eq:setting-X}, we obtain
\[
\begin{aligned}
\mathrm DJ_\Sigma(K)[\Delta K]
&=\operatorname{tr}(\Delta P\Sigma)\\
&=2\operatorname{tr}\!\left(X_{K,\Sigma}(RK-B^\top P_K)^\top\Delta K\right).
\end{aligned}
\]
Comparison with the Frobenius inner product yields~\eqref{eq:setting-grad}.

\subsection{Compact sublevels and convergence under full-rank covariance}\label{app:full-rank}

Set $s_0=\lambda_{\min}(\Sigma)>0$, $q_0=\lambda_{\min}(Q)>0$, and $r_0=\lambda_{\min}(R)>0$. Matrix norms $\|\cdot\|_2$ in this subsection are spectral norms. The inequality $e^{F_Kt}e^{F_K^\top t}\succeq e^{-2\|F_K\|_2t}I_n$ gives
\[
J_\Sigma(K)\ge
\frac{s_0(nq_0+r_0\|K\|_F^2)}{2(\|A\|_2+\|B\|_2\|K\|_F)}.
\]
This lower bound tends to infinity as $\|K\|_F\to\infty$. For bounded stabilizing gains approaching the stability boundary, $\mathfrak m(K)\downarrow0$, and
\[
J_\Sigma(K)\ge s_0q_0\int_0^\infty\|e^{F_Kt}\|_F^2\,\mathrm dt
\ge\frac{s_0q_0}{2\mathfrak m(K)}\longrightarrow\infty.
\]
The second inequality uses $\|e^{F_Kt}\|_F\ge\rho(e^{F_Kt})=e^{-\mathfrak m(K)t}$, where $\rho$ is the spectral radius. Hence every nonempty cost sublevel is a compact subset of $\mathcal S$.

On the compact trajectory sublevel, the gradient and Hessian of the cost are bounded. Thus $f(\tau):=\|\nabla J_\Sigma(K(\tau))\|_F^2$ has bounded derivative and is uniformly continuous. By~\eqref{eq:energy}, $f$ is integrable on $[0,\infty)$. A nonnegative, uniformly continuous, integrable function tends to zero, so every accumulation point of the trajectory is stationary. The stationary-point uniqueness established in the proof of Proposition~\ref{prop:full-rank} then gives $K(\tau)\to K_\star$.

\subsection{Local coordinates near an unexcited boundary point}
\label{app:collision-chart}

We prove the rank assertion underlying \eqref{eq:collision-chart}. Write $c_B:=\bar\ell^\top B\in\mathbb R^{1\times m}$. Since $U^\top\bar\ell=0$, every column $U_j$ of $U$ belongs to the stable subspace $\mathcal E:=\ker\bar\ell^\top$. The restriction $F_s:=F_{\bar K}|_{\mathcal E}$ is invertible, so $w_j:=F_s^{-1}U_j$ is well defined. The vectors $w_1,\ldots,w_s$ are linearly independent and lie in $\mathcal E$, whereas $\bar\ell^\top\bar v=1$ implies $\bar v\notin\mathcal E$.

Differentiating the left and right eigenvector equations in an arbitrary gain direction $\Delta K$ gives
\begin{equation*}
 Da_j(\bar K)[\Delta K]=c_B\Delta K w_j,
 ~~ Dy(\bar K)[\Delta K]=c_B\Delta K\bar v.
\end{equation*}
The first identity follows by multiplying the differentiated left eigenvector equation by $w_j$; the term involving $Dy$ vanishes because $\bar\ell^\top w_j=0$. The second is the differential formula for a simple eigenvalue. Since $c_B\ne0$, the row $c_B\Delta K$ ranges over all of $\mathbb R^{1\times n}$. The independence of $w_1,\ldots,w_s,\bar v$ therefore gives rank $s+1$ for $(Da,Dy)$. Completing these differentials with fixed linear gain coordinates and applying the inverse function theorem yields $\chi$ and its inverse $\kappa$.

Treating the system data and the full-column-rank factor $U$ as parameters, the parameter-dependent inverse function theorem also gives smoothly varying coordinates under small perturbations. The equations $a=y=0$ retain nearby solutions. Hence the unexcited boundary point may move with the data rather than remain at the original gain.

\subsection{Local trapping near a strictly attracting direction}
\label{app:collision-trap}

We supply the trapping argument for Theorem~\ref{thm:collision-criterion}. By \eqref{eq:boundary-gradient} and the smooth coordinate change, \eqref{eq:desingularized} has the form
\[
 \begin{aligned}
 r_\sigma&=\mathcal F(r,z)+yF_1(r,y,z),\\
 y_\sigma&=y\bigl(G(r,z)+yG_1(r,y,z)\bigr),\\
 z_\sigma&=yZ_1(r,y,z),~~ \tau_\sigma=y,
 \end{aligned}
\]
where the remainder functions are smooth on bounded direction regions. The implicit function theorem gives a smooth branch $r_*(z)$ satisfying $\mathcal F(r_*(z),z)=0$. Set $w:=r-r_*(z)$ and $A_*(z):=D_r\mathcal F(r_*(z),z)$. Since $A_*(z_*)=-\nabla_r^2\Psi_{z_*}(r_*)$ is Hurwitz, a sufficiently small fixed neighborhood satisfies
\[
 w_\sigma=A_*(z)w+O(\|w\|_2^2)+O(y),~~
 y_\sigma\le-\gamma y,~~ \|z_\sigma\|_2\le C_0y
\]
for some $\gamma,C_0>0$. Choose $P\succ0$ solving $A_*(z_*)^\top P+PA_*(z_*)=-I$. Continuity and Young's inequality give
\[
 \frac{\mathrm d}{\mathrm d\sigma}(w^\top Pw)
 \le-c_0\|w\|_2^2+C_1y^2.
\]
For a sufficiently large fixed $L>0$ and a smaller neighborhood, the function $\mathcal V:=w^\top Pw+Ly^2$ satisfies $\mathcal V_\sigma\le-c_1\mathcal V$ for some $c_1>0$. This controls $w$ and $y$, while
\[
 \int_0^\infty\|z_\sigma\|_2\,\mathrm d\sigma
 \le\frac{C_0y(0)}{\gamma}
\]
controls the tangential drift. Choose the initial $z$ strictly inside the neighborhood and take $y(0)$ small enough that this bound is less than its distance to the boundary. Escape in every direction is then excluded.

The smoothly extended system consequently exists for all $\sigma\ge0$, with $y(\sigma)>0$ at every finite $\sigma$. Moreover, $y\to0$, $w\to0$, and $z\to z_\infty$ for some $z_\infty$. The endpoint in the original time satisfies
\[
 T=\int_0^\infty y(\sigma)\,\mathrm d\sigma
 \le\frac{y(0)}{\gamma}<\infty.
\]
The strict initial inequalities define a nonempty open set. Since $(r,y,z)\mapsto(yr,y,z)$ is a diffeomorphism for $y>0$, its image in the original gain space is also open. All estimates admit strict margins and therefore persist under sufficiently small smooth perturbations of the local coefficients and coordinates. When there are no tangential variables, all conditions involving $z$ are omitted.

Finally, $r\to r_\infty:=r_*(z_\infty)$ and $y'\to G(r_\infty,z_\infty)<0$. Integrating $y'$ up to the collision time gives
\begin{equation}
 y(\tau)\sim[-G(r_\infty,z_\infty)](T-\tau).
 \label{eq:app-margin-rate}
\end{equation}

\subsection{Planar calculations and the construction at arbitrary rank}
\label{app:collision-construction}

We first verify the coefficients of the planar example. Put $c=5\sqrt3$, $b=\sqrt{14}$, and $D=c^2+b^2+1=90$. Write the gain as $K=(-k,b+\zeta)$ and set $Y:=\zeta-ck$. The closed-loop characteristic polynomial is $\lambda^2+(1+Y+ck)\lambda+Y$. Hence, on bounded direction regions where $k=O(Y)$, the true spectral margin satisfies $y=Y+O(Y^2)$. Direct solution of the planar Lyapunov equation gives
\begin{equation}
 J_2(k,Y)=\frac12+
 \frac{ck+k^2+Dk^2/Y}{2(1+Y+ck)}.
 \label{eq:app-planar-cost}
\end{equation}
At the boundary, $v=(c,1)^\top/\sqrt{76}$ and $\ell=\sqrt{76}e_2$, so $q=45/76$. Eigenvector differentiation gives $L=(-\sqrt{76},0)$, and eigenvalue differentiation gives $d=(c,1)^\top$.

To compute the gradient of the smooth remainder, note that at $k=0$ the response from $e_1$ is $(e^{-t},0)^\top$. Thus $J_2(0,Y)=1/2$ and the critical-mode pole term vanishes identically, giving $h_Y(0,0)=0$. For each fixed $Y>0$, the first derivative of the pole term with respect to $k$ vanishes at $k=0$. Equation~\eqref{eq:app-planar-cost} therefore yields $h_k(0,Y)=c/[2(1+Y)]$. Letting $Y\downarrow0$ gives $g=(-c/2,0)^\top$ in the original gain coordinates.

Let $R_0:=k/Y$. Equation~\eqref{eq:app-planar-cost} gives the boundary direction field and margin velocity
\[
 \begin{aligned}
 \mathcal F_0(R_0)
 &=-\frac D2\left[(1+c^2)R_0^3+3cR_0^2
                 +(2+c^2/D)R_0+c/D\right],\\
 G_0(R_0)
 &=\frac D2\left[c^2/D+2cR_0+(1+c^2)R_0^2\right].
 \end{aligned}
\]
At $R_{0,*}=-1/(4\sqrt3)$, direct substitution yields
\[
 \mathcal F_0(R_{0,*})=0,~~
 \mathcal F_0'(R_{0,*})=-\frac{15}{4},~~
 G_0(R_{0,*})=-\frac{15}{4}.
\]
At the boundary, the spectral direction variable satisfies $r=\sqrt{76}R_0$. Thus $r_*=-\sqrt{57}/6$ and $\Psi''(r_*)=15/4$, proving \eqref{eq:planar-direction}.

We now give the full construction for Theorem~\ref{thm:open-collision}. Choose distinct $\lambda_3,\ldots,\lambda_n\in(2,3)$ and start with
\[
 \begin{aligned}
 A&=A_2\oplus\operatorname{diag}(-\lambda_3,\ldots,-\lambda_n),\\
 B^{(0)}&=(e_2,0,\ldots,0),~~ Q=I_n,~ R=I_m,\\
 U&=(e_1,e_3,\ldots,e_{s+1}),~~ \Sigma=UU^\top.
 \end{aligned}
\]
Let the first row of $\bar K$ be $(0,b,0,\ldots,0)$ and all other rows be zero. The closed-loop eigenvalues are $0,-1,-\lambda_3,\ldots,-\lambda_n$, and $\bar\ell=\sqrt{76}e_2$, so \eqref{eq:unexcited-controllable} holds. Empty lists and sums are omitted when $n=2$ or $s=1$.

To calculate the direction linearization, first set every gain row after the first to zero and write
\[
 K_{1,:}=(-k,b+\zeta,\xi_3,\ldots,\xi_n),~~ Y=\zeta-ck.
\]
For initial state $e_j$, $j\ge3$, the component $x_j(t)=e^{-\lambda_jt}$ generates planar state and control responses linear in $\xi_j$. Consequently,
\begin{equation}
 J_\Sigma=J_2(k,Y)+\sum_{j=3}^{s+1}
 \left(\frac1{2\lambda_j}+\xi_j^2T_j(k,Y)\right).
 \label{eq:app-extra-cost}
\end{equation}
The cost is independent of $\xi_j$ for the unexcited indices $j>s+1$. The $j$th component of the critical left eigenvector satisfies $\ell_j=-(\ell_{\mathrm{pl}}^\top e_2)\xi_j/(\lambda_j-y)$, where $\ell_{\mathrm{pl}}$ consists of its first two components. At the boundary, $q(\ell_{\mathrm{pl}}^\top e_2)^2=D/2$, so the pole decomposition gives $YT_j(k,Y)\to D/(2\lambda_j^2)$. Thus the leading-order cost needed for the direction linearization is
\[
 f_0+\frac c2 k+\frac D2\frac{k^2}{Y}
 +\sum_{j=3}^{s+1}\frac{D}{2\lambda_j^2}\frac{\xi_j^2}{Y},
\]
where $f_0$ is constant. The matrix governing gradient dynamics induced by the original Frobenius metric in the $(k,Y)$ coordinates is
\[
 \begin{pmatrix}1&-c\\-c&1+c^2\end{pmatrix}.
\]
It is the identity in the added $\xi_j$ coordinates, with no cross block between these two parts.

Set $R_j:=\xi_j/Y$. At $R_j=0$, the planar direction remains $R_{0,*}$ and the margin velocity remains $G_*=-15/4$. The additional direction equations are
\[
 (R_j)_\sigma=
 \left(-\frac{D}{\lambda_j^2}-G_*\right)R_j
 +O(Y)+O(\|R-R_*\|_2^2).
\]
By \eqref{eq:app-extra-cost}, the added gains affect the planar direction field only at quadratic or higher order, so there is no corresponding first-order coupling. The additional eigenvalues are precisely the negative values in \eqref{eq:extra-direction-eigen}. Unexcited gain components and the remaining input rows are tangential variables. At the construction point, the other input columns vanish, and the corresponding gain rows enter the instantaneous cost quadratically. They therefore leave the direction linearization unchanged. At the boundary, the spectral coordinates $a/y$ and $(R_0,R_3,\ldots,R_{s+1})$ are related by an invertible linear transformation; the two margin coordinates and time rescalings agree to leading order. Hence the direction Jacobian in the original gain space has the same strictly negative eigenvalues.

To make $(A,B)$ controllable, perturb the first input column to $B_{:,1}^{(\eta)}=e_2+\eta\sum_{j=3}^n e_j$, with $\eta>0$, leaving the others zero. The original $\bar K$ still satisfies $a=y=0$: the second closed-loop row remains zero, and the restriction to $\{x_2=0\}$ remains diagonal with eigenvalues $-1,-\lambda_3,\ldots,-\lambda_n$. The eigenvalues of $A$ are distinct. Its left eigenvector for $-1$ has a nonzero second component; those for $b$ and $-\lambda_j$ are $e_2^\top$ and $e_j^\top$, respectively. Each has a nonzero product with $B_{:,1}^{(\eta)}$, so the PBH (Popov--Belevitch--Hautus) criterion~\cite[Theorem 13.3]{rugh1996} gives controllability of $(A,B^{(\eta)})$. The strict collision criterion persists for sufficiently small $\eta$; this perturbation is unnecessary when $n=2$.

Next replace the first column of $U$ by $U_1=e_1+\delta\sum_{j=s+2}^n e_j$, with $\delta>0$, leaving its other columns unchanged. All columns still belong to $\{x_2=0\}$, so the boundary conditions are preserved. The stable restriction has distinct eigenvalues, and each left eigenvector has a nonzero product with $U$. Thus the excitation pair on that restriction is controllable. The strictly attracting direction and its nonzero value persist for sufficiently small $\delta$.

In a sufficiently small collision basin, the stable modes remain excited and the critical-mode left excitation satisfies $a=yr\ne0$. The PBH criterion then gives controllability of $(F_K,U)$ at every time before collision, and hence $X_{K,\Sigma}\succ0$.

Finally, system controllability, controllability of the excitation pair on the stable restriction, and positive definiteness of the weights are open conditions. By appendix~\ref{app:collision-chart}, general small data perturbations still admit a nearby boundary point with $a=y=0$, where the strictly attracting direction and negative margin velocity persist. After shrinking the collision basin, the same argument preserves Gramian positivity. The map $U\mapsto UU^\top$ from full-column-rank factors has local smooth sections on the fixed-rank positive-semidefinite manifold. Openness in factor space therefore gives relative openness in the data space with covariance rank $s$.

\subsection{Linear degeneration of the state Gramian}
\label{app:gramian-rate}

Consider a sufficiently small collision basin from Theorem~\ref{thm:open-collision}. The trapping estimates give $r\to r_\infty\ne0$, $z\to z_\infty$, and $y\to0$; equation~\eqref{eq:app-margin-rate} gives the linear asymptotic behavior of the margin in the remaining optimization time.

Set $e(K):=\ell(K)/\|\ell(K)\|_2$ and complete it to an orthonormal frame using a smooth orthonormal basis of $\ker\ell(K)^\top$. The Gramian integral and the left eigenvector equation give the exact identity
\begin{equation}
 e(K)^\top X_{K,\Sigma}e(K)
 =\frac{\ell(K)^\top\Sigma\ell(K)}{2y\|\ell(K)\|_2^2}
 =\frac{\|r\|_2^2}{2\|\ell(K)\|_2^2}\,y.
 \label{eq:app-gramian-normal}
\end{equation}
In this orthogonal decomposition, the other diagonal block converges to the positive definite Gramian on the stable subspace at the boundary. Its positivity follows from controllability of the excitation pair on the stable restriction in the construction.

The off-diagonal block is $O(y)$. Indeed, separate the semigroup into its critical spectral projection and its uniformly exponentially decaying stable part. The pure critical-mode Gramian term is $O(y)$ because $\|U^\top\ell\|_2=O(y)$; the cross terms between the critical and stable modes are also $O(y)$; and the stable part has zero left projection onto $\ell$. Hence the unique small eigenvalue differs from the scalar in \eqref{eq:app-gramian-normal} by $O(y^2)$, while the other eigenvalues remain uniformly bounded away from zero. Letting $r\to r_\infty$ and $\ell\to\ell_\infty$ proves \eqref{eq:gramian-asymptotic}.

\subsection{Uniform hitting estimates for the local square-root bounds}\label{app:local-passage}

We give the direction, tangential-drift, and initial-condition estimates used in Theorem~\ref{thm:local-sqrt}. Let $(r_*,z_*)$ be the strictly attracting direction and set $A_*:=D_r\mathcal F(r_*,z_*)$. By \eqref{eq:collision-criterion}, $A_*$ is Hurwitz. Choose $P_r\succ0$ satisfying $A_*^\top P_r+P_rA_*=-I$, and define $W(r):=(r-r_*)^\top P_r(r-r_*)$. For sufficiently small $\rho>0$, the direction field points strictly into the ellipsoid $\mathcal E:=\{r:W(r)<\rho^2\}$ at its boundary. By decreasing the tangential radius $R_z>0$, we can choose $a_0,\gamma>0$ such that
\begin{equation}
 \begin{aligned}
 \nabla W(r)^\top\mathcal F(r,z)&\le-2a_0,
 &&r\in\partial\mathcal E,\\
 G(r,z)&\le-2\gamma,
 &&r\in\overline{\mathcal E},
 \end{aligned}
 ~~ \|z-z_*\|\le2R_z.
 \label{eq:app-passage-signs}
\end{equation}
When there are no tangential variables, omit the conditions involving $z$.

Choose a simple-mode neighborhood $V$ with compact closure, and take $y_b>0$ small enough that
\[
 \left\{\kappa(yr,y,z):
 r\in\overline{\mathcal E},\ 0\le y\le y_b,\
 \|z-z_*\|\le3R_z/2\right\}
\]
is contained in $V$ and the local barrier holds throughout $V$. All estimates below are uniform on this region of bounded ratio variables, and we take $y_b\le1$.

Let $b_0(z):=b(\kappa(0,0,z))$ and $d(z):=\nabla_Ky(\kappa(0,0,z))$. Differentiate \eqref{eq:regularized-pole} with respect to the gain and then substitute $a=yr$ to obtain
\begin{equation}
 K'=-p(r,z)+\frac{\varepsilon}{y^2}b_0(z)d(z)
            +O\!\left(y+\frac{\varepsilon}{y}\right).
 \label{eq:app-regularized-gain}
\end{equation}
Here $p$ is the leading gradient of the unperturbed cost defined in Section~\ref{sec:collisions}. The local coordinate map and its derivatives are bounded on $\overline V$. Thus \eqref{eq:app-regularized-gain} gives
\[
 r_\sigma=\mathcal F(r,z)+O(y)+O(\varepsilon/y^2),
 ~~ y'=G(r,z)+O(y)+O(\varepsilon/y^2),
\]
and $z_\sigma=O(y)+O(\varepsilon/y)$. These estimates retain all tangential components of the gain gradient.

First decrease $y_b$ so that the $O(y)$ terms preserve the strict signs in \eqref{eq:app-passage-signs}, and then choose a large fixed $L_*\ge1$. On the truncated region $y\ge L_*\sqrt\varepsilon$, we have $\varepsilon/y^2\le L_*^{-2}$, so we can ensure that
\[
 \nabla W(r)^\top r_\sigma<0~(r\in\partial\mathcal E),
 ~~ y'\le-\gamma.
\]
Moreover, $\varepsilon/y\le y/L_*^2$, hence $\|z'\|\le M_z$ for a constant $M_z$ independent of $\varepsilon$. Choose $0<y_{\mathrm{in}}<y_b/2$ so that $M_zy_{\mathrm{in}}/\gamma<R_z/2$, and define
\[
 \mathcal C:=\left\{\kappa(yr,y,z):
 r\in\mathcal E,\ 0<y<y_{\mathrm{in}},\ \|z-z_*\|<R_z\right\}.
\]
The map $(r,y,z)\mapsto\kappa(yr,y,z)$ is a local diffeomorphism for $y>0$, so $\mathcal C$ is nonempty and open.

Fix $K_0\in\mathcal C$ and take $\varepsilon$ sufficiently small that $L_*\sqrt\varepsilon<y(K_0)$. Before reaching $y=L_*\sqrt\varepsilon$, the trajectory cannot cross the direction boundary outward, while $y'\le-\gamma$ rules out exit through the upper boundary. This stage lasts at most $y(K_0)/\gamma$ in the original optimization time, so
\[
 \|z(\tau)-z(0)\|
 \le M_z\tau
 \le\frac{M_zy(K_0)}\gamma<R_z/2.
\]
Thus the trajectory cannot first exit through the tangential boundary either. For each fixed positive $\varepsilon$, the closure of the truncated region is a compact subset of the stable domain. The local solution can therefore be continued until it reaches $y=L_*\sqrt\varepsilon$, which must occur in finite time.

The local barrier prevents a downward crossing of $y=c_*\sqrt\varepsilon$ during the entire first stay in $V$, including any portion after the trajectory has left the direction region above. The hitting point supplies the local upper bound, and the barrier supplies the lower bound. All direction and drift constants are uniform on a fixed region. For every compact set $\mathcal C_0\subset\mathcal C$, we also have $\inf_{K_0\in\mathcal C_0}y(K_0)>0$, so $\varepsilon_0$ can be chosen uniformly. This proves all the hitting properties required in Theorem~\ref{thm:local-sqrt}.

\subsection{Bounds on the worst-case initial-state cost}\label{app:worst-cost}

Apply the same spectral separation used for the cost pole to the matrix integral $P_K=\int_0^\infty e^{F_K^\top t}W_Ke^{F_Kt}\,\mathrm dt$. The contribution of the critical mode to itself is
\[
 \int_0^\infty e^{-2y(K)t}\Pi(K)^\top W_K\Pi(K)\,\mathrm dt
 =\frac{q(K)}{y(K)}\ell(K)\ell(K)^\top.
\]
The remaining terms form a symmetric matrix $\mathcal R_P(K)$ that extends smoothly across $y=0$. On $\overline V$, the quantity $q(K)\|\ell(K)\|_2^2$ has a strictly positive lower bound and a finite upper bound, while $\|\mathcal R_P(K)\|_2$ is bounded. Hence there are positive constants $c_1,C_1,C_2$ such that
\begin{equation*}
 \frac{c_1}{y(K)}-C_2
 \le\lambda_{\max}(P_K)
 \le\frac{C_1}{y(K)}+C_2,
 ~~ K\in V\cap\mathcal S.
\end{equation*}
The margin lower bound in Theorem~\ref{thm:local-sqrt} bounds the right-hand side by a constant multiple of $\varepsilon^{-1/2}$. At the hitting point supplied by Appendix~\ref{app:local-passage}, we have $y=L_*\sqrt\varepsilon$, so the left-hand side is bounded below by a positive constant multiple of $\varepsilon^{-1/2}$ for sufficiently small $\varepsilon$. This proves \eqref{eq:worst-cost}.

\subsection{The control-energy estimate and the high-gain spectral limit}\label{app:high-gain}

Let $P=P_K\succ0$. Left-multiplying the Lyapunov equation \eqref{eq:setting-P} by $P^{-1}$ and taking traces yields
\[
 2\operatorname{tr}F_K+\operatorname{tr}(P^{-1}Q)
      +RKP^{-1}K^\top=0.
\]
Applying the Cauchy--Schwarz inequality to each column of $U$ gives
\[
 \|KU\|_F^2
 \le(KP^{-1}K^\top)\operatorname{tr}(U^\top PU)
 =(KP^{-1}K^\top)J_{\Sigma_0}(K).
\]
Discarding the positive term $\operatorname{tr}(P^{-1}Q)$ proves \eqref{eq:control-energy-bound}.

Now consider a sequence of stabilizing gains satisfying $J_{\Sigma_0}(K_j)\le M$ and $\|K_j\|_F\to\infty$. In the coordinates $T=[U,B]$ and the block notation from the main text, the energy estimate gives $v_j\to+\infty$ and $\|u_j\|_2=O(\sqrt{v_j})$. To verify the first assertion, note that stability implies $v_j>\operatorname{tr}A$. If $v_j$ were bounded above on a subsequence, the energy estimate would also bound $u_j$ there, contradicting $KT=(u,v)$ and the invertibility of $T$.

Set $\kappa_j:=v_j-\alpha\to+\infty$. The block determinant identity gives
\begin{equation}
 \begin{aligned}
 \frac{\det(\lambda I-T^{-1}F_{K_j}T)}{\kappa_j}
 ={}&\left(1+\frac{\lambda}{\kappa_j}\right)\det(\lambda I-Z)\\
 &-\frac{(\beta-u_j)\operatorname{adj}(\lambda I-Z)\omega}{\kappa_j},
 \end{aligned}
 \label{eq:app-high-gain-polynomial}
\end{equation}
where $\operatorname{adj}$ denotes the adjugate. Since $\|u_j\|_2/\kappa_j\to0$, the right-hand side converges uniformly on every compact subset of the complex plane to $\det(\lambda I-Z)$. Choose disjoint small closed contours around the distinct eigenvalues of $Z$ and apply Rouch\'e's theorem~\cite{conway1978} to \eqref{eq:app-high-gain-polynomial}. Exactly $n-1$ closed-loop eigenvalues then converge, with algebraic multiplicity, to the spectrum of $Z$. The remaining eigenvalue is real and equals the closed-loop trace $\operatorname{tr}Z-\kappa_j$ minus the sum of these $n-1$ bounded eigenvalues; hence it tends to $-\infty$.

The operator $Z_0$ and its matrix representation $Z$ have the same spectrum and are Hurwitz. Therefore,
\[
 \lim_{j\to\infty}\mathfrak m(K_j)
 =-\max_{\lambda\in\operatorname{spec}(Z_0)}\operatorname{Re}\lambda>0.
\]
This completes the exclusion of high-gain escape in Lemma~\ref{lem:localization}.

\subsection{Numerical integration and error control}\label{app:numerical-settings}

Near the stability boundary, we analytically separate the critical-mode pole before solving the remaining Lyapunov equations to reduce errors from nearly singular linear systems. We use the time reparameterization
\[
\frac{\mathrm dK}{\mathrm d\sigma}
=-\mathfrak m(K)\nabla_KJ_{\Sigma_\varepsilon}(K),
~~
\frac{\mathrm d\tau}{\mathrm d\sigma}=\mathfrak m(K),
\]
which preserves the gain trajectory of the Euclidean flow in the original Frobenius metric. The near-boundary phase uses DOP853 with relative and absolute tolerances $10^{-9}$ and $10^{-11}$, respectively. After a positive-excitation trajectory first crosses margin $0.08$ outward, we resume integration in the original optimization time using Radau with tolerances $2\times10^{-9}$ and $2\times10^{-11}$, terminating at~\eqref{eq:numerical-stop}.

The minimum margin is obtained by searching the integrator's continuous interpolant and refining each detected local minimum. We retain two records for the three-dimensional trajectory without added excitation: Table~\ref{tab:collision-records} uses the original record stopped at margin $10^{-6}$, while Figure~\ref{fig:experiments}A--C uses a denser record extended to margin $10^{-7}$. All positive-excitation runs use the same full-precision initial gain, independent of $\varepsilon$.

\end{document}